\documentclass[11pt]{amsart}

\usepackage{hyperref}
\usepackage[margin=1in]{geometry}
\usepackage{amsmath,amssymb,amsthm, xcolor, url}
\usepackage[colorinlistoftodos, textsize=tiny, textwidth=80]{todonotes}
\usepackage{mathrsfs}
\usepackage{thm-restate}
\usepackage[normalem]{ulem}
  
\theoremstyle{plain}
\newtheorem{theorem}{Theorem}[section]
\newtheorem{proposition}[theorem]{Proposition}
\newtheorem{lemma}[theorem]{Lemma}
\newtheorem{corollary}[theorem]{Corollary}
\theoremstyle{definition}
\newtheorem{definition}[theorem]{Definition}
\newtheorem{remark}[theorem]{Remark}

\theoremstyle{remark}

\numberwithin{equation}{section}
\numberwithin{theorem}{section}

\DeclareMathOperator\Disc{Disc}

\DeclareMathOperator\Res{Res}
\DeclareMathOperator\Gal{Gal}

\newcommand\Q{\mathbb{Q}}

\newcommand\calN{\mathcal N}
\newcommand\calD{\mathcal D}
\newcommand\calO{\mathcal O}

\newcommand\scg{\mathcal P}

\renewcommand\O{\mathcal O}
\newcommand\ol{\overline}

\newcommand\NnK[2]{N_{n, K}(#1; #2)}
\newcommand\NnQ[2]{N_{n, \mathbb{Q}}(#1; #2)}
\newcommand\NdQ[3]{N_{#1, \mathbb{Q}}(#2; #3)}
\newcommand\NdK[3]{N_{#1, K}(#2; #3)}
\newcommand\GnK[2]{\scg_{n,K}(#1; #2)}

\newcommand\FnK[2]{\mathcal{F}_{n,K}(#1; #2) }

\title{
Improved lower bounds for the number of fields with alternating Galois group
}

\author{Aaron Landesman}
\address{Department of Mathematics, Stanford University, Stanford, CA 94305}
\email{aaronlandesman@stanford.edu}

\author{Robert J. Lemke Oliver}
\address{Department of Mathematics, Tufts University, Medford, MA 02155}
\email{robert.lemke{\_{}}oliver@tufts.edu}

\author{Frank Thorne}
\address{Department of Mathematics, University of South Carolina, Columbia, SC 29208}
\email{thorne@math.sc.edu}

\def\listtodoname{List of Todos}
\def\listoftodos{\@starttoc{tdo}\listtodoname}

\begin{document}

\begin{abstract}
	Let $n \geq 6$ be an integer.  We prove that the number of number fields with Galois group $A_n$ and absolute discriminant at most $X$ is asymptotically at least $X^{1/8 + O(1/n)}$.
	 For $n \geq 8$
	 this improves upon the previously best known lower bound of $X^{(1 - \frac{2}{n!})/(4n - 4) - \epsilon}$, due to Pierce, Turnage-Butterbaugh, and Wood.
\end{abstract}
\maketitle

\section{Introduction}

For any number field $K$, any integer $n \geq 2$, any real number $X$, and any transitive subgroup $G$ of the symmetric group $S_n$, let 
\begin{equation}\label{eq:defF}
	\FnK{G}{X} := \{ L/K: [L:K]=n, \mathrm{Gal}(\widetilde{L}/K) \simeq G, \ | \calN_{K/\Q}(\calD_{L/K}) | \leq X \},
\end{equation}
where $\widetilde{L}$ denotes the Galois closure of $L$ over $K$, $\mathcal D_{L/K}$ is the relative discriminant of $L$ over $K$, and $\calN_{K/\Q}$ denotes the norm map.  
Define
\begin{equation}\label{eq:defN}
\NnK{G}{X} := \#\FnK{G}{X}.
\end{equation}  
Our main result is the following bound on $\NnK{A_n}X$, the number of extensions whose Galois closure is the alternating group.

\begin{theorem}\label{thm:an-lower-bound}
Let $n \geq 6$ be an integer, with $n \neq 7$, and let $A_n$ denote the alternating group on $n$ elements.  
For any number field $K$, as $X\to \infty$, we have 
\[
	\NnK{A_n}{X} \gg \begin{cases} X^{\frac{(n-4)(n^2 - 4)}{8(n^3 - n^2)}}
 & \text{if } n \text{ is even}, \\ X^{\frac{(n - 7)(n + 2)}{8n^2}} & \text{if } n \text{ is odd}. \end{cases}
\]

\end{theorem}

The proof, given in \autoref{section:proof}, is inspired by Ellenberg and Venkatesh's lower bounds \cite{EllenbergVenkatesh} on $\NnK{S_n}{X}$.
Their strategy was to count degree $n$ polynomials with a height bound on the coefficients, sufficient to guarantee that these polynomials generate fields whose discriminant
is bounded as in \eqref{eq:defF}.
Hilbert's irreducibility theorem guarantees that almost all of them generate $S_n$-extensions, and Ellenberg and Venkatesh applied estimates from the geometry of numbers 
to bound the multiplicity with which each number field was thus counted.

What we do here is to import this machinery into Hilbert's original construction of $A_n$-polynomials \cite[pp. 126-127]{Hilbert}. Indeed, 
\cite{Hilbert}
is the very paper in which Hilbert proved his irreducibility theorem.
Let $r := \lfloor \frac{n-1}{2} \rfloor$.
Following Hilbert, we construct a univariate polynomial $F$ over a function field $K(a_1,\dots,a_r,t)$ whose Galois group over 
	$K(a_1, \ldots, a_r,t)$	
	is $A_n$. As Hilbert argued, there are infinitely many specializations of the $a_i$ and $t$ for which the resulting polynomial has Galois group $A_n$ over $K$,
	and following Ellenberg and Venkatesh we quantify how many such specializations there are and how many $A_n$-fields they produce. 

\subsection{Discussion of the main result}
The asymptotic behavior of $\NnK{A_n}{X}$ is known only when $n=3$.  In particular, if $K$ does not contain third roots of unity, then $\NdK{3}{A_3}{X} \sim c_K X^{1/2}$ for a positive constant $c_K$, while if $K$ does contain third roots of unity, then $\NdK{3}{A_3}{X} \sim c_K X^{1/2} \log X.$  When $K= \mathbb{Q}$, this follows from Cohn \cite{cohn}, and for general $K$ from Wright \cite[Theorem 1.1]{wright}, who proved an asymptotic formula for $\NnK{G}{X}$ whenever $G$ is abelian.  For $n\geq 4$, as yet unproved cases of Malle's conjecture \cite{malle}
predict an asymptotic formula for $\NnK{A_n}{X}$.  When $n=4$, this prediction states that $\NdK{4}{A_4}{X} \sim c_K X^{1/2} (\log X)^2$ if 
$K$ contains third roots of unity, and that $\NdK{4}{A_4}{X} \sim c_K X^{1/2} \log X$ if 
$K$ does not contain third roots of unity.
For all $n \geq 5$ and all number fields $K$ Malle's conjecture predicts that $\NnK{A_n}{X} \sim c_K X^{1/2} \log X$. When $n \geq 5$, Malle's prediction does not depend on whether $K$ contains third roots of unity.

For lower bounds on $\NnK{A_n}{X}$, Baily \cite{baily} proved that $\NdQ{4}{A_4}{X} \gg X^{1/2}$. For $n > 4$,
Pierce, Turnage-Butterbaugh, and Wood \cite[Theorem 2.6]{PTBW} proved that $\NnQ{A_n}{X} \gg X^{\beta_n - \epsilon}$ with
\begin{equation}\label{eq:PTBW}
\beta_n := \frac{1 - \frac{2}{n!}}{4n - 4}.
\end{equation}
This result is stronger than ours for $n = 6$, 
and is, to our knowledge, the only known quantitative lower bound for $n = 5$ and $n=7$. 
There is no theoretical
obstruction to our method working in the case $n = 7$, but in this case it
yields trivial results.
For $n \geq 8$, our results improve upon those of \cite{PTBW}, and are, to our knowledge, the only bounds stated in the literature for $K \neq \mathbb{Q}$.

Upper bounds on $\NnK{A_n}{X}$ are also known. Indeed, these will be an ingredient in our proof. 
For $n = 4$ and $n = 5$ the sharpest known bounds are those in \cite{BSTTTZ} and \cite{BCT} 
respectively. For $n \geq 6$ we have the Schmidt
bound \cite{schmidt} 
\begin{equation}\label{eq:schmidt}
	\NnK{G}{X} \ll X^{\frac{n + 2}{4}},
\end{equation} 
which holds for arbitrary subgroups $G \subset S_n$. When $G = A_n$ and $K=\mathbb{Q}$, Larson and Rolen \cite{LarsonRolen} obtained an upper bound that is smaller by a factor of about $X^{1/4}$.
For large $n$, the Schmidt bound 
was improved by Ellenberg and Venkatesh \cite{EllenbergVenkatesh} 
to $\NnK{G}{X} \ll X^{\exp(C \sqrt{\log n})}$ with a constant $C$ that may be made explicit.
This upper bound was further improved to
$\NnQ{G}{X} \ll X^{C \log^3 n}$ for a suitable constant $C$
by 
Couveignes
in 
\cite{couveignes2019enumerating}
and then to 
$\NnQ{G}{X} \ll X^{C \log^2 n}$ by the second and third authors \cite{LOT}; further,
by counting degree $n$ extensions $L/K$ as degree $n[K:\Q]$ extensions over $\Q$,
one obtains bounds for $\NnK{G}{X}$ with $C$ now depending on $[K : \Q]$.


\begin{remark}
It is possible to slightly improve \autoref{thm:an-lower-bound} under certain hypotheses that are known for large $n$.  See \autoref{prop:best-possible} at the end of this note.
\end{remark}

Our methods also allow us to obtain a lower bound for the number of $S_n$ extensions with a fixed quadratic resolvent.  Recall that the quadratic resolvent of an $S_n$ extension $L/K$ is the unique quadratic extension $M/K$ contained in the normal closure $\tilde{L}/K$.  For any quadratic extension $M/K$, define
\[
	\NnK{S_n,M}{X}
		:= \#\{ L \in \FnK{S_n}{X} : M \subseteq \tilde{L}/K\}.
\]

\begin{theorem}\label{thm:sn-lower-bound}
Let $n \geq 6$ be an integer, with $n \neq 7$. For any number field $K$ and any quadratic extension $M/K$, as $X\to \infty$, we have 
\[
	\NnK{S_n,M}{X} \gg \begin{cases} X^{\frac{(n-4)(n^2 - 4)}{8(n^3 - n^2)}}
 & \text{if } n \text{ is even}, \\ X^{\frac{(n - 7)(n + 2)}{8n^2}} & \text{if } n \text{ is odd}. \end{cases}
\]
\end{theorem}


\subsection{Acknowledgements}
This work was conceived at the Palmetto Number Theory Series conference in Columbia, SC in December, 2018. We 
thank Matthew Boylan and Michael Filaseta for organizing the conference, and the National Science Foundation (under Grant No. DMS-1802259) 
and the National Security Agency
(under Grant No. H98230-16-1-0247) for funding it. 
We also thank an anonymous referee for a particularly careful reading and for suggesting numerous improvements.

AL was supported by the National Science Foundation Graduate Research Fellowship Program under Grant No. DGE-1656518.  RJLO was supported by National Science Foundation grant, DMS-1601398.  FT was supported by grants from the Simons Foundation (Nos. 563234 and 586594).

\section{Preliminaries}

\subsection{Notation}
\label{subsection:notation}
We begin by fixing a base number field $K$ with $[K : \Q] = d$ and an integer $n \geq 2$.
differ from line to line.
Throughout, any constants implied by the notation $\gg$, $\ll$, and $O(-)$ will be allowed to depend on these quantities.

For the purposes of this paper, we define the {\itshape height} of 
a monic polynomial $f := x^n + c_1 x^{n-1} + \dots + c_n \in \calO_K[x]$ by
\begin{equation}\label{def:height}
\mathrm{ht}(f) := \max \|c_i\|^{1/i},
\end{equation}
where, as in \cite{EllenbergVenkatesh}, for any algebraic number $\alpha$ we write $\|\alpha\|$ for 
the largest Archimedean valuation of $\alpha$.

Finally, it will be convenient to introduce a parameter $Y$, depending on $X, d,$ and $n$, given by
\begin{equation}
Y := X^{1/dn(n-1)}.
	\label{y-defn}
\end{equation}

\subsection{Discriminants and resultants}
For the convenience of the reader, we next review some well known facts about discriminants and resultants.
\begin{definition}
	\label{definition:resultant}
	Let $R$ be an integral domain ring with fraction field $K(R)$ whose algebraic closure
	is denoted $\overline{K(R)}$.
	Given polynomials 
$f := c_0 x^n + \dots + c_n \in R[x]$ and $g := b_0 x^m + \dots + b_m \in R[x]$, 
	the {\em resultant} $\mathrm{Res}(f,g)$ of $f$ and $g$ is defined by
\begin{align}\label{eq:resultant}
	\mathrm{Res}(f,g) & := c_0^m b_0^n \prod_{\substack{\alpha,\beta \\ f(\alpha)=0 \\ g(\beta)=0}} (\alpha-\beta) \\ \label{eq:result}
	& = (-1)^{nm} b_0^n \prod_{g(\beta) = 0} f(\beta),
	\end{align}
	where the product runs over roots $\alpha$ of $f$ and roots $\beta$ of $g$ in $\overline{K(R)}$, counted with multiplicity.  

The {\em discriminant} of $f$ is
\begin{equation}\label{definition:discriminant}
\Disc(f) := \frac{(-1)^{\frac{n(n-1)}{2}}}{c_0}\Res(f,f'),
\end{equation}
where $f'$ is the derivative of $f$.
\end{definition}

\begin{lemma}\label{lem:disc-res}
	Let $R$ be an integral domain with fraction field $K(R)$ whose algebraic closure is denoted $\overline{K(R)}$. Let $f = c_0 x^n + \dots + c_n \in R[x]$ be a polynomial with roots $\alpha_1, \ldots, \alpha_n$ in $\overline {K(R)}$.  Then
\[
\mathrm{Disc}(f) = 
(-1)^{n(n-1)/2} n^n c_0^{n-1} \prod_{\beta: f^\prime(\beta)=0} f(\beta)
= c_0^{2n-2} \prod_{1 \leq i < j \leq n}(\alpha_i - \alpha_j)^2
\]
\end{lemma}
\begin{proof}
	The first equality follows from \eqref{eq:result}.
The second is a straightforward calculation factoring $f$ as a product of
linear polynomials over $\overline{K(R)}$, as is explained in \cite[Proposition IV.8.5]{Lang}.
Note that the last expression in the statement of the lemma is used as the definition of the discriminant in \cite[Proposition IV.8.5]{Lang}.
\end{proof}
For convenience, we also note two easy consequences of \eqref{eq:resultant} and \eqref{eq:result}.  
\begin{corollary}
	\label{corollary:switch-resultant}
	With notation as in \autoref{definition:resultant},
	$\mathrm{Res}(f,g) = (-1)^{mn} \mathrm{Res}(g,f)$. In particular $\mathrm{Res}(f,g) = \mathrm{Res}(g,f)$ if either $f$ or $g$ has even degree, as holds in the case $g = f^\prime$. 
\end{corollary}
\begin{corollary}
	\label{corollary:sum-resultant}
	For $R$ an integral domain, and $f,g,h \in R[x]$, we have $\mathrm{Res}(f + hg, g) = \mathrm{Res}(f,g)$ and $\Res(f, gh) = \Res(f, g) \Res(f, h)$.
\end{corollary}

\begin{remark}
	\label{remark:}
	From \eqref{definition:discriminant} we see that 
the discriminant of $f$ is invariant under the Galois group permuting the roots
of $f$, and hence can be expressed as a weighted
homogeneous polynomial in the coefficients of $f$.
Further, if $f$ is monic with $\mathrm{ht}(f) \ll Y$, then
\autoref{lem:disc-res} implies that $\| \mathrm{Disc}(f) \| \ll Y^{n(n-1)}$.  
\end{remark}

\subsection{Preliminaries on number field counting}

We now import some of the machinery initially developed by Ellenberg and Venkatesh 
\cite{EllenbergVenkatesh} to bound $N_{n,K}(S_n;X)$ from below that was further studied by the second and third authors \cite{LemkeOliverThorne}.

Our strategy for proving \autoref{thm:an-lower-bound}, following the proof of the lower bound in \cite[Theorem 1.1]{EllenbergVenkatesh} given in 
	\cite[\S3]{EllenbergVenkatesh}, is to reduce the problem
to proving a lower bound for the number of algebraic integers with small norm that generate $A_n$-extensions of $K$.
To formulate this reduction, for a transitive subgroup $G \subseteq S_n$, we write
\begin{equation}\label{def:gk}
\GnK{G}{Y} := \{ z \in \calO_{\overline{K}} \ : \ \|z\| \leq Y, \ [K(z) : K] = n, \ \Gal(\widetilde{K(z)}/K) \simeq G  \}. 
\end{equation}
We begin by quoting a bound on the multiplicity with which a given extension $L/K$ is cut out by elements of $\GnK{G}{Y}$.

\begin{lemma}\label{lem:multiplicity}
Let $L/K$ be a degree $n$ extension and let $G = \mathrm{Gal}(\widetilde{L}/K)$.  Let 
\[
M_{L/K}(Y) := \#\{ z \in \GnK{G}{Y} : K(z) \simeq L\}.
\]
Then $M_{L/K}(Y) \ll \max\{ Y^{nd} |\mathrm{Disc}(L)|^{-1/2}, Y^{nd/2}\}$ where $\mathrm{Disc}(L)$ is the absolute discriminant of $L$.
\end{lemma}
\begin{proof}
This is essentially \cite[Proposition 7.5]{LemkeOliverThorne}, and we recall the proof.

We count $z \in \calO_{L}$ with $\|z\| \leq Y$, with no reference to $K$. To do so, embed $\mathcal{O}_L \hookrightarrow \mathbb{R}^{nd}$ as usual and
write $\lambda_0, \dots, \lambda_{nd - 1}$ for the successive minima (with $\lambda_0 \leq 1$). If $\lambda_{nd - 1} \leq Y$, then an integral basis for $\mathcal{O}_L$ fits inside a box
of side length $O(Y)$, so that $M_{L/K}(Y) \ll Y^{nd} |\mathrm{Disc}(L)|^{-1/2}$. Otherwise, let $k < nd - 1$ be the largest integer for which $\lambda_k \leq Y$, and we have
\begin{equation}\label{eq:lo75}
M_{L/K}(Y) \ll \frac{Y^{k + 1}}{\lambda_1 \dots \lambda_k} \ll \frac{Y^{k + 1}}{\Disc(L)^{1/2}} \lambda_{k + 1} \cdots \lambda_{nd - 1}.
\end{equation}
If $k \leq \frac{nd}{2} - 1$ then the first half of \eqref{eq:lo75} yields $M_{L/K}(Y) \ll Y^{nd/2}$; if $k > \frac{nd}{2} - 1$ then we use the second inequality 
in \eqref{eq:lo75} in combination with the bounds $Y < \lambda_{nd -1}$ and  $\lambda_{nd - 1} \ll \Disc(L)^{1/nd}$ \cite[Theorem 3.1]{BSTTTZ} to again conclude that
\[
M_{L/K}(Y) \ll 
\Disc(L)^{1/2} \bigg( \frac{Y}{\Disc(L)^{1/nd}} \bigg)^{k + 1} \ll Y^{nd/2}.
\]
\end{proof}

With \autoref{lem:multiplicity} in hand, we are able to make explicit the reduction from counting integers generating $G$-extensions to counting $G$-extensions themselves.

\begin{proposition}\label{prop:reduction} 
	Let $K$ be a number field of degree $d$, and let $G$ be a transitive subgroup of $S_n$ for some $n \geq 3$. 
	Let $e\geq 1/(n-1)$ and $C>n$ be constants such that:
	\begin{itemize}
		\item $\NnK{G}{X} \ll X^e$, and
		\item $\#\GnK{G}{Y} \gg Y^{dC}$
	\end{itemize}
	hold for all sufficiently large $X$ and $Y$.  Then
	\begin{equation}\label{eq:reduction}
		\NnK{G}{X} 
			\gg X^{\frac{C-n/2}{n^2-n}}
	\end{equation}
    if $C \geq n\left(e + \frac12\right)$, and
    \[
		\NnK{G}{X}
			\gg X^{\frac{2e(C-n)}{(2e-1)(n^2-n)}}
    \]
    if not.
\end{proposition}
\begin{proof}
	Let $Y = X^{1/dn(n-1)}$. We will proceed with the proof in three cases, depending on whether $C \geq n \left( e+\frac{1}{2} \right)$ and whether $e \geq 1/2$.
	In each of these cases, we will show the existence of some $Z>1$ such that
	\begin{equation}\label{eqn:reduction-goal}
		\sum_{L \in \FnK{G}{Z}} M_{L/K}(Y)
			\leq \#\GnK{G}{Y} / 2.
	\end{equation}
	Since $|\mathrm{Disc}(K(\alpha))| \ll Y^{dn(n-1)}$ for any $\alpha \in \GnK{G}{Y}$, we have $\NnK{G}{X} \gg \#\GnK{G}{Y} / M$, where $M$ is the maximum of $M_{L/K}(Y)$ over $G$-extensions $L/K$ with norm of discriminant greater than $Z$.  This maximum $M$ may be estimated by means of \autoref{lem:multiplicity}, while \eqref{eqn:reduction-goal} may be established by combining \autoref{lem:multiplicity} with the assumed upper bound on $\NnK{G}{X}$.  The different cases in the statement of the proposition follow by making suitable choices of $Z$, as we now explain.
		
Suppose first that $e > 1/2$.  Then for any $Z \leq Y^{nd}$, by \autoref{lem:multiplicity} and partial summation, 
		we have
	\[
		\sum_{L \in \FnK{G}{Z}} M_{L/K}(Y)
			\ll Y^{nd} Z^{e-\frac{1}{2}}.
	\]
	If $C \leq n\left(e + \frac12\right)$, then we choose $Z$ to be a sufficiently small multiple of $Y^{\frac{2d(C-n)}{(2e-1)}}$ for which \eqref{eqn:reduction-goal} holds. 
	If $C \geq n\left(e + \frac12\right)$ we choose $Z$ to be a sufficiently small multiple of $Y^{nd}$.
	
	If instead $e \leq 1/2$, then we again
	have $C \geq n\left(e + \frac12\right)$ and
	 take $Z$ to be a sufficiently small multiple of $Y^{nd}$. In this case our hypotheses imply that
	\[
		\sum_{L \in \FnK{G}{Y^{nd}}} M_{L/K}(Y)
			\ll Y^{nd} \log Y,
	\]
	the $\log Y$ factor being relevant only for $e = \frac12$.
\end{proof}

Using the Schmidt bound of \eqref{eq:schmidt}
in \autoref{prop:reduction}, i.e. taking $e = \frac{n+2}{4}$, we obtain the following immediate consequence.

\begin{corollary}\label{cor:reduction-schmidt}
With the assumptions of \autoref{prop:reduction},
\[
	\NnK{G}{X} \gg 
		\begin{cases}
			X^{ \frac{C-n/2}{n^2-n} } & \text{if } C \geq \frac{n^2+4n}{4}, \text{ and} \\
			X^{ \frac{(C-n)(n+2)}{n^3-n^2} } & \text{if } C \leq \frac{n^2+4n}{4}.
		\end{cases}
\]
\end{corollary}

\section{Proof of \autoref{thm:an-lower-bound}}
\label{section:proof}

\subsection{Overview of proof}
In this section, we prove \autoref{thm:an-lower-bound}. The proof in the even case is given in \autoref{subsection:even-proof} and completed in \autoref{subsubsection:even-proof},
while the proof in the odd case is given in \autoref{subsection:odd-proof} and completed in \autoref{subsubsection:odd-proof}.
As described earlier, our strategy is to adapt the original constructions of $A_n$-polynomials 
by Hilbert \cite[pp. 126-127]{Hilbert}, and count the number of distinct fields thus produced. For $K$ a number field, we construct a polynomial $F \in K(a_1,\dots,a_r,t)[x]$ whose Galois group over the function field
	$K(a_1, \ldots, a_r,t)$	
	is $A_n$.
By specializing the variables $a_1, \dots, a_r$ and $t$ suitably, we will obtain many irreducible polynomials 
whose Galois groups are still $A_n$, and then we use \autoref{prop:reduction} to conclude the proof.

The constructions differ depending on whether $n$ is even or odd.  One may consult 
\cite[\S10.3]{Serre} for an English-language treatment of Hilbert's work in the case that $n$ is even,
and \cite[\S5.2]{sellares:realizing-sn-and-an-as-galois-groups}
for a treatment of both the even and odd cases. 

\subsection{Notation for proof}
We begin by introducing some notation which will be used in both cases. 
Set $r = \frac{n}{2} - 1$ when $n$ is even and set $r = \frac{n - 1}{2}$ if $n$ is odd.
We then introduce a polynomial
\begin{equation}\label{eq:h-definition}
h(x) = x^r + a_1 x^{r - 1} + \cdots + a_r \in \O_K[a_1, \cdots, a_r][x]
\end{equation}
in $x$ and in the indeterminates $a_i$, and define $g(x) \in \O_K[a_1, \cdots, a_r, a][x]$ by
\begin{equation}
	g(x) := \begin{cases} n(x - a) h(x)^2 & \text{ if $n$ is even,}\\ (n - 1)(x - a) h(x)^2 & \text{ if $n$ is odd.} \end{cases}
	\label{eq:g-definition}
\end{equation}
For $\alpha_1, \cdots, \alpha_r, \alpha, \tau \in K$ we denote
by $|_{\alpha_1, \ldots, \alpha_r, \alpha, \tau}$
	the evaluation map
	\begin{align*}
		|_{\alpha_1, \ldots, \alpha_r, \alpha, \tau} \colon K[a_1, \ldots, a_r,a,t,x] & \rightarrow K[x] \\
		f(a_1, \ldots, a_r, a, t, x) & \mapsto f(\alpha_1, \ldots, \alpha_r, \alpha, \tau, x) =: f|_{\alpha_1, \ldots, \alpha_r, \alpha, \tau}.
	\end{align*}
	We also use analogous notation when the domain has fewer indeterminates; for
	example, $|_{\tau}$ denotes the map $K[t, x] \rightarrow K[x]$ given by $f(t,x) \mapsto f(\tau,x)$. Observe also that when 
	$\alpha_1, \cdots, \alpha_r, \alpha, \tau \in \O_K$, these maps restrict to homomorphisms from the appropriate polynomial rings over $\O_K$ to $\O_K[x]$.

\subsection{Proof of \autoref{thm:an-lower-bound} and
\autoref{thm:sn-lower-bound} for even $n$}
\label{subsection:even-proof}

Assume that $n \geq 6$ is even. 
Based on Hilbert's construction \cite[p. 125-126]{Hilbert}, we consider polynomials whose derivative is nearly a square.
Recall our notation for $g(x)$ as defined in \eqref{eq:g-definition}.
Let $\tilde{f}(x) \in K(a_1, \ldots, a_r,a)[x]$ denote the antiderivative of $g(x)$ with respect to $x$ such that $(x-a)^2$ divides $\tilde{f}(x)$. 

Then, for each $\gamma \in \frac{1}{n!} \O_K[a_1, \ldots, a_r, a][t]$, 
define
\[
\widetilde{f_\gamma}(x) := \widetilde{f}(x) + \gamma, \ \ \ 
f_\gamma(x) :=  (n!)^n\tilde{f_\gamma}(x/n!).
\]
Note for $\gamma \in \frac{1}{n!} \O_K[a_1, \ldots, a_r,a][t]$, we will have
$f_\gamma(x) \in \O_K [a_1, \ldots, a_r, a,t][x]$ is monic with integral coefficients.
\begin{lemma}
	\label{lemma:even-disc-square}
	With notation as above, the discriminant of $f_\gamma(x)$, viewed as a polynomial in $x$, is a square if and only if $(-1)^{n/2}\gamma$ is a square. 
\end{lemma}
\begin{proof}
	Using \autoref{lem:disc-res}, since $n$ is even, it is equivalent to compute whether the discriminant of $\widetilde{f}_\gamma(x)$ is a square.
	Using \autoref{lem:disc-res} again, we find
\begin{align*}
	\mathrm{Disc}(\widetilde{f}_\gamma(x))
	&= (-1)^{n(n-1)/2}n^n \prod_{\beta : \widetilde{f}^\prime(\beta)=0} \widetilde{f}_\gamma(\beta) \\
	&= (-1)^{n/2}n^n \widetilde{f}_\gamma(a) \left(\prod_{\beta : h(\beta) =0} \widetilde{f}_\gamma(\beta)\right)^2 \\
	&= (-1)^{n/2}n^n \gamma \cdot \left(\prod_{\beta : h(\beta) =0} \widetilde{f}_\gamma(\beta)\right)^2.
\end{align*}
The final expression is a square if and only if $(-1)^{n/2} \gamma$ is a square.
\end{proof}

We now essentially recall the construction of Hilbert on which ours is based; Hilbert in fact further takes $\delta = 1$ in the lemma below. 
In \eqref{eq:g-definition}, further specialize to the case 
\[
a = 0, \ \ \ h(x) := (x-\beta_1) \cdots (x-\beta_r),
\]
where $\beta_1, \dots, \beta_r$ are nonzero and distinct elements of $\O_K$, for which 
$\widetilde{f}(\beta_1), \dots, \widetilde{f}(\beta_r)$ are also nonzero and distinct. (For example, choose $\beta_i = i$ for each $i$.
Then $g(x)$ is nonnegative for $x > 0$, so that $\widetilde{f}$ is increasing there.)

We write $\widetilde{P}$, $P$, $\widetilde{P}_\gamma$, and $P_\gamma$ for the associated specializations of $\widetilde{f}$, $f$, $\widetilde{f}_\gamma$ and $f_\gamma$;
the first two are elements of $K[x]$, and the latter two of $K[x,t]$. 

\begin{lemma}\label{lemma:even-hilbert}
	With the notation above, for any $\delta \in K^\times$, the Galois group $G$ of $\widetilde{P}_{ (-1)^{n/2} \delta t^2} \in K(t)[x]$ over $K(t)$ is $A_n$ if $\delta$ is a square and $S_n$ otherwise.  If $\delta$ is not a square, then the Galois group of $\widetilde{P}_{ (-1)^{n/2} \delta t^2}$ over $K(\sqrt{\delta})(t)$ is $A_n$.
\end{lemma}
\begin{proof}
See \cite[p. 125-126]{Hilbert}, \cite[\S10.3, Theorem]{Serre}, or \cite[\S5.2]{sellares:realizing-sn-and-an-as-galois-groups}; we summarize
Mart\'inez's argument from \cite{sellares:realizing-sn-and-an-as-galois-groups}.

We begin by considering the Galois group $G^\prime$ of $\widetilde{P}_{ (-1)^{n/2} \delta t^2}$ as a polynomial over $\mathbb{C}(t)$.  The discriminant of $\widetilde{P}_{ (-1)^{n/2}\delta t^2} \in \mathbb{C}(t)[x]$ is a square  by
\autoref{lemma:even-disc-square}, 
so $G^\prime \subseteq A_n$.  We will show that $G^\prime$ is in fact all of $A_n$, for which it suffices to show that $G^\prime$ is generated by $3$-cycles and that it is transitive.  We prove these in turn.

Since there are no unramified finite extensions of $\mathbb{C}(t)$, $G^\prime$ is generated by the inertia groups at the ramified primes. 
It therefore suffices to show that these are
all $3$-cycles.

By our discriminant computation in \autoref{lemma:even-disc-square}, the ramified primes are given by $(t)$ and $\left(t\sqrt{\delta} \pm \sqrt{(-1)^{n/2 + 1} \widetilde{P}(\beta_i)} \right)$.
Modulo $(t)$, $\widetilde{P}_{ (-1)^{n/2} \delta t^2}$ has a double root at $x = 0$ and its other roots are simple.
(Note that $\widetilde{P}(\beta_i) \neq 0$ for each $\beta_i$, as the derivative of $g$ is nonnegative.) Therefore, the inertia group at $(t)$ is either trivial or generated by a transposition; since it is a subgroup of $A_n$, it must be trivial.

Modulo $\left(t\sqrt{\delta} \pm \sqrt{(-1)^{n/2 + 1} \widetilde{P}(\beta_i)} \right)$, $\widetilde{P}_{ (-1)^{n/2} \delta t^2}$ has a triple root at $x = \beta_i$ and its other roots are simple.
The corresponding inertia group is therefore either trivial or a $3$-cycle, and this completes the proof that $G^\prime$ is generated by $3$-cycles.

	To complete the proof that $G^\prime \simeq A_n$, we will verify $G^\prime$ acts transitively on the $n$ roots of $\widetilde{P}_{ (-1)^{n/2} \delta t^2}$ over an algebraic closure.
Notice that $\widetilde{P}_{ (-1)^{n/2} \delta t^2} = \widetilde{P} +(-1)^{n/2}\delta t^2$.  As $\widetilde{P}$ has a simple root, $(-1)^{n/2+1}\delta \widetilde{P}$ is not a square in $\mathbb{C}(x)$.  Thus, $\widetilde{P}_{ (-1)^{n/2} \delta t^2}$ is irreducible as a polynomial in $\mathbb{C}(x)[t]$.  We conclude that $G^\prime = A_n$.

It thus follows for any finite extension $L/K$ that the Galois group of $\widetilde{P}_{ (-1)^{n/2} \delta t^2}$ over $L(t)$ contains $A_n$, and is thus equal to either $A_n$ or $S_n$.  The remainder of the claim follows from observing that, by \autoref{lemma:even-disc-square}, these two possibilities correspond exactly to whether or not $\delta$ is a square in $L$.
\end{proof}

\begin{remark}
Our proof corrects a sign error, found not only in 
\cite{sellares:realizing-sn-and-an-as-galois-groups}
but also in 
\cite{Hilbert}. At least in Mart\'inez's case, this can be traced to 
a missing sign in 
\cite[Observation 5.2]{sellares:realizing-sn-and-an-as-galois-groups}. As Lang remarks in his {\itshape Algebra} \cite{Lang}: ``Serre once pointed out to me that the sign $(-1)^{n(n - 1)/2}$ was missing
in the first edition of this book, and that this sign error is quite common in the literature, occurring as it does in the works of van der Waerden, Samuel, and Hilbert."
\end{remark}

We now extrapolate Hilbert's result to prove an analogue over a larger base field.

\begin{lemma}
	\label{lemma:even-big-galois}
	With notation as above, for any $\delta \in K^\times$, the Galois group of $f_{(-1)^{n/2}\delta t^2}$ over $K(a_1,\dots,a_r,a,t)[x]$ is $A_n$ if $\delta$ is a square and $S_n$ otherwise.  If $\delta$ is not a square, then the Galois group of $f_{(-1)^{n/2}\delta t^2}$ over $K(\sqrt{\delta})(a_1,\dots,a_r,a,t)[x]$ is $A_n$.
	\end{lemma}
\begin{proof}
	As a first step, note that the Galois group of $f_\gamma$ agrees with that of $\widetilde{f}_\gamma$, and so we will compute the Galois group of the latter polynomial
	in the case that 
	$\gamma = (-1)^{n/2}\delta t^2$.

In the case that $\delta$ is a square, by \autoref{lemma:even-hilbert}, the polynomial $\widetilde{f}_\gamma$ specializes to a polynomial
$\widetilde{F}_\gamma \in K(t)[x]$ with Galois group $A_n$ over $K(t)$.
Hence, the Galois group of 
$\widetilde{f}_{(-1)^{n/2}\delta t^2}$ over $K(a_1,\dots,a_r,a,t)[x]$ contains $A_n$. Since the discriminant of 
this polynomial is a square
by \autoref{lemma:even-disc-square}, its Galois group must be exactly $A_n$.  The case that $\delta$ is not a square follows analogously.
\end{proof} 

\subsubsection{Completing the proof of \autoref{thm:an-lower-bound} and \autoref{thm:sn-lower-bound}
for even $n$}
\label{subsubsection:even-proof}

We begin with \autoref{thm:an-lower-bound}, for which we make the choice $\delta =1$.  Using \autoref{lemma:even-big-galois}, 
we may choose $\gamma \in K[a_1, \ldots, a_r, a][t]$ so that $f_\gamma$ has Galois group $A_n$.
We vary $\alpha_1, \ldots, \alpha_r, \alpha, \tau \in \calO_K$ subject to the constraints
\begin{align*}
	\|\alpha_i\| \ll Y^i, \|\alpha\| \ll Y, \|\tau\| \ll Y^{n/2},
\end{align*}
making a total of $\asymp_K (Y \cdot Y^2 \cdots Y^{n/2-1} \cdot Y \cdot Y^{n/2})^d = Y^{\frac{d(n^2+2n+8)}{8}}$ choices of the parameters.
By the Hilbert irreducibility theorem (\autoref{thm:hit}) we have
\begin{equation}\label{eqn:even-poly-count}
\#\{ \alpha_1, \dots, \alpha_r, \alpha, \tau \in \O_K : \mathrm{ht}(f_\gamma|_{\alpha_1, \ldots, \alpha_r, \alpha, \tau}) \ll Y, \mathrm{Gal}(f_\gamma|_{\alpha_1, \ldots, \alpha_r, \alpha, \tau} / K) \simeq A_n\}
	\gg Y^{\frac{d(n^2+2n+8)}{8}}.
\end{equation}

We now note that, for each fixed polynomial $q \in K[x]$, there are 
at most $\deg \gamma = 2$ many tuples $(\alpha_1, \ldots, \alpha_r, \alpha, \tau)$ so that $f_\gamma|_{\alpha_1, \ldots, \alpha_r, \alpha, \tau}$
coincides with $q$, or equivalently so that $\widetilde{f_\gamma}|_{\alpha_1, \ldots, \alpha_r, \alpha, \tau}$ coincides with $\widetilde{q}$, where $\widetilde{q}(x) := (n!)^{-n}q(n!x)$.
To see why, first note that the
value $\alpha$ is determined as the unique root of $\frac{\partial \widetilde q}{\partial x}$ which appears with odd multiplicity. 
Then, because $h$ is monic and we know the value of $h^2$, the values $\alpha_1, \ldots, \alpha_r$ are determined.
Having determined the values $\alpha, \alpha_1, \ldots, \alpha_r$, there are then at most $\deg \gamma$ (viewed as a polynomial in $t$) many choices of $\tau$ so that the constant coefficient of
$\widetilde{f_\gamma}|_{\alpha_1, \ldots, \alpha_r, \alpha, \tau}$, viewed as a polynomial in $x$, is equal to
the constant coefficient of $\widetilde{q}$.

Therefore, in the notation of \eqref{def:gk} we have $\GnK{A_n}{Y} \gg_K Y^{\frac{d(n^2 + 2n + 8)}{8}}$, and hence by taking $C=(n^2+2n+8)/8$ in \autoref{cor:reduction-schmidt}
we conclude that 
\[
\NnK{A_n}{X} 
	\gg X^{\frac{(n - 4)(n^2 - 4)}{8(n^3 - n^2)}}.
\]
The proof of \autoref{thm:sn-lower-bound} for even $n$ is exactly the same, except choosing $\delta$ to be any integral element for which $M = K(\sqrt{\delta})$.
\subsection{Proof of \autoref{thm:an-lower-bound} 
	and \autoref{thm:sn-lower-bound}
for odd $n$}
\label{subsection:odd-proof}

Assume that $n\geq 7$ is odd. Again we follow Hilbert \cite[p. 126 - 127]{Hilbert};
see also \cite[\S5.2]{sellares:realizing-sn-and-an-as-galois-groups} for a version in English.
In the case that $n$ was even, we considered polynomials whose derivative is nearly a square, and used this to compute the Galois group of the resulting polynomial.
In the case $n$ is odd, we instead consider polynomials for which $x \frac{\partial f}{\partial x}-f$ is nearly a square.  
Using properties of resultants, this will let us control the discriminant of $f$ in much the same way as the case that $n$ is even.

Set $r=(n-1)/2$, define $g$ and $h$ as in \eqref{eq:h-definition} and \eqref{eq:g-definition}, and define $\ol g(x), \ol h(x) \in \O_K[a_1, \ldots, a_{r-1}, a][x]$
to be the polynomials obtained from $g$ and $h$ by replacing $a_r$ with $2a_{r - 1}a$.

\begin{lemma}
	\label{lemma:height-bound}
	Given the notation above, there is a unique polynomial $\widetilde{f}(x) \in {\frac{1}{n!} \O_K}[a_1, \ldots, a_{r-1},a][x]$
	of degree $n$, necessarily monic,
	satisfying 
	\[ x \frac{\partial \widetilde{f} }{\partial x} - \widetilde{f}(x) = \ol{g}(x), \ \ \ \widetilde{f}'(0) = 0.
	\]
	For each $\gamma \in \frac{1}{n!}\O_K[a_1, \ldots, a_{r-1}, a][t]$, the polynomial
	$\widetilde{f}_\gamma(x) := \widetilde{f}(x) + \gamma \cdot x \in \frac{1}{n!} \O_K[a_1, \ldots, a_{r-1},a,t][x]$
	is a solution to $x\frac{\partial \widetilde{f}_\gamma }{\partial x} - \widetilde{f}_\gamma(x) = \ol{g}(x)$ in $K[a_1, \ldots, a_{r-1},a,t][x]$.
	Lastly, for any $\alpha_1, \ldots, \alpha_{r-1}, \alpha, \tau \in \O_K$
with 
\begin{equation}\label{eq:bounds}
	\alpha \ll Y, \ \ \  \tau \ll Y^{\frac{n - 1}{\deg \gamma}}, \ \ \ \alpha_i \ll Y^i \ (1 \leq i \leq r -1),
\end{equation}
$\widetilde{f}_\gamma(x)|_{\alpha_1, \ldots, \alpha_{r-1},\alpha,\tau}$ 
has height $\ll_\gamma Y$.
\end{lemma}
\begin{proof}
	To show there is a solution to the equation $x \frac{\partial \widetilde{f} }{\partial x} - \widetilde{f}(x) = \ol{g}(x)$,
by differentiating both sides, it suffices to show there is a solution to the equation
$x \frac{\partial^2 \widetilde{f} }{\partial x^2} =  \frac{\partial \ol{g}}{\partial x}$.
For this, we only need check that $\frac{\partial \ol{g}}{\partial x}$ is divisible by $x$.
This holds precisely because the coefficient of $x$ in $g(x)$ is
$(n-1)(a_r^2 - 2a_{r-1}aa_r)$ and the image of this coefficient in the quotient
$\O_K[a_1, \ldots, a_r, a][x]/(a_r - 2a_{r-1} a)$
is $0$.

Since $\frac{\partial^2 \widetilde{f} }{\partial x^2}$ is then uniquely determined, all terms of $\widetilde{f}(x)$ except the linear and constant terms are determined.
The constant term is determined by the equation $x \frac{\partial \widetilde{f} }{\partial x} - \widetilde{f}(x) = \ol{g}(x)$.
Then, any linear term will satisfy the above equation, but the condition $\widetilde{f}'(0) = 0$ uniquely determines the linear term to be $0$. 
Then we find
$x \frac{\partial \widetilde{f}_\gamma }{\partial x} - \widetilde{f}_\gamma(x)  = 
x  \frac{\partial \widetilde{f}}{\partial x} - \widetilde{f}(x) = \ol{g}(x)$.

Monicity of $\widetilde{f}$ follows from the differential equation defining it and the assumption that the leading coefficient of $\overline{g}$ is $n-1$.

Observe that $f(x)$ is a polynomial of degree $n$ and the coefficients of $\frac{\partial^2 \widetilde{f}}{\partial x^2}$ lie in $\O_K[a_1, \ldots, a_{r-1}, a,t]$ since $x\frac{\partial^2 \widetilde{f} }{\partial x^2} = \frac{\partial \ol{g}}{\partial x}$.
It follows that the coefficients of $\widetilde{f}_\gamma$ lie in $\frac{1}{n!} \O_K[a_1, \ldots, a_{r-1}, a,t]$.

Finally, the bound on the height of $\widetilde{f}_{\gamma}(x)|_{\alpha_1, \ldots, \alpha_{r-1},\alpha,\tau}$
follows straightforwardly from expanding $\widetilde{f}_{\gamma}(x)|_{\alpha_1, \ldots, \alpha_{r-1},\alpha,\tau}$.
\end{proof}

Now, for the remainder of the subsection, with $\widetilde{f}$ and $\widetilde{f}_\gamma$
as in \autoref{lemma:height-bound},
define
\begin{align*}
	f(x) &:= n!^n \widetilde{f}(x/n!) \in \O_K[a_1, \ldots, a_{r-1}, a, t][x], \\
	f_\gamma(x) &:= n!^n \widetilde{f}_\gamma(x/n!) \in \O_K[a_1, \ldots, a_{r-1}, a, t][x].
\end{align*}

\begin{lemma}
	\label{lemma:odd-disc-square}
	With notation as above, $\mathrm{Disc}(f_\gamma(x))$ is a square if and only if 
	$(-1)^r\frac{\partial \widetilde{f}_\gamma }{\partial x}(a)$
is a square.
\end{lemma}
\begin{proof}
By \autoref{lem:disc-res}, the discriminant of $f_\gamma$ is a square if and only if the discriminant of $\widetilde{f}_\gamma$ is a square. 
Therefore, we will show the discriminant of $\widetilde{f}_\gamma$ is a square if and only if $(-1)^r\frac{\partial \widetilde{f}_\gamma }{\partial x}(a)$ is. 
Indeed, we have
\begin{align*}
	\mathrm{Disc}(\widetilde{f}_\gamma )
	&= (-1)^{n(n-1)/2} \mathrm{Res}\left(\widetilde{f}_\gamma ,\frac{\partial \widetilde{f}_\gamma }{\partial x}\right) &\text{by Definition }\ref{definition:resultant}\\
	&= (-1)^{r} \mathrm{Res}\left(x\frac{\partial \widetilde{f}_\gamma }{\partial x}- \ol{g}(x),\frac{\partial \widetilde{f}_\gamma }{\partial x}\right) \\
	&= (-1)^{r} \mathrm{Res}\left(\ol{g}(x),\frac{\partial \widetilde{f}_\gamma }{\partial x}\right) &\text{by Corollary }\ref{corollary:sum-resultant}\\
	&= (-1)^{r} \mathrm{Res}\left(\ol{h}(x),\frac{\partial \widetilde{f}_\gamma }{\partial x}\right)^2 \cdot
	\mathrm{Res}\left(\frac{\partial \widetilde{f}_\gamma }{\partial x}, (n - 1)(x - a)\right) 
	&\text{by Corollaries }\ref{corollary:switch-resultant}\text{ and }\ref{corollary:sum-resultant}\\
	&= (-1)^{r} 
	 \left( \prod_{\ol{h}(\beta) = 0} \frac{\partial \widetilde{f}_\gamma }{\partial x}(\beta) \right)^2  \cdot
	(n - 1)^{n - 1}  \frac{\partial \widetilde{f}_\gamma }{\partial x}(a).
	&\text{by \eqref{eq:result}}
\end{align*}
\end{proof}

The analogue of \autoref{lemma:even-hilbert} is the following.

\begin{lemma}\label{lem:odd-specialization}
Let $\widetilde{P} \in \Q[x]$ be of odd degree $n$, satisfying the differential equation
\[
x\frac{\partial \widetilde{P}}{\partial x} - \widetilde{P} = \overline{G}, \ \ \ \ \overline{G} = (n - 1)(x - \alpha)(x - \beta_1)^2 (x - \beta_2)^2 \cdots (x - \beta_r)^2,
\]
where the $\beta_i$ are all distinct positive rational numbers,
and $2\alpha\sum_i \beta_i^{-1} := -1$.
Then, for a suitable choice of the $\beta_i$,
the values $\frac{\partial \widetilde{P}}{\partial x}(\beta_1), \ldots, \frac{\partial \widetilde{P}}{\partial x}(\beta_r), \frac{\partial \widetilde{P}}{\partial x}(\alpha)$ are
pairwise distinct.

Further, for any number field $K$ and any $\delta \in K^\times$, setting $\gamma :=  (-1)^r \delta t^2 -\frac{\partial \widetilde{P}}{\partial x}(\alpha)$, 
the Galois group of splitting field of $\widetilde{P}_\gamma := \widetilde{P} + \big((-1)^r \delta t^2 - \frac{\partial \widetilde{P}}{\partial x}(\alpha)\big)x$ over $K(t)$ is $A_n$ if $\delta$ is a square and $S_n$ otherwise.  If $\delta$ is not a square, then the Galois group of $\widetilde{P}_\gamma$ over $K(\sqrt{\delta})(t)$ is $A_n$.
\end{lemma}

\begin{proof}
	First, we check $\frac{\partial \widetilde{P}}{\partial x}(\beta_1), \ldots, \frac{\partial \widetilde{P}}{\partial x}(\beta_r)$ 
		are pairwise distinct, following \cite[p. 127]{Hilbert}. 
By assumption that all $\beta_i > 0$, we find $\alpha < 0$ and hence $\overline{G}(x) > 0$ for $x$ positive. 	
Thus, for $\beta_i > \beta_j$ we have
\[
\frac{\partial \widetilde{P}}{\partial x}(\beta_i)
- 
\frac{\partial \widetilde{P}}{\partial x}(\beta_j)
= \frac{\widetilde{P}(\beta_i)}{\beta_i} - \frac{\widetilde{P}(\beta_j)}{\beta_j} 
= 
\int_{\beta_j}^{\beta_i} \frac{x \frac{\partial \widetilde{P}}{\partial x} - \widetilde{P}(x)}{x^2} dx
=\int_{\beta_j}^{\beta_i} \frac{\overline{G}(x)}{x^2} dx
> 0.
\]

We now check that, for a suitable choice of the $\beta_i$, we have $\frac{\partial \widetilde{P}}{\partial x}(\alpha) \neq
 \frac{\partial \widetilde{P}}{\partial x}(\beta_i)$ for each $i$. To do this, formally set $\beta_i = 1$ for each $i$, in which case we
 have $\widetilde{P}(x) = (x - 1)^n - nx - 1$. Then $\frac{\partial \widetilde{P}}{\partial x}(\beta_i) = \frac{\partial \widetilde{P}}{\partial x}(1) = -n$ is negative,
 while 
 \[
 \frac{\partial \widetilde{P}}{\partial x}(\alpha) = \frac{\partial \widetilde{P}}{\partial x}\left(\frac{-1}{n - 1}\right) =
 n \left[ \left(\frac{-n}{n - 1}\right)^{n - 1} - 1\right] > 0
 \]
 is positive. By continuity, these signs will persist if the $\beta_i$ are perturbed slightly, so that it suffices to take
 the $\beta_i$ all sufficiently close to $1$.

We next check the Galois group of the splitting field of $\widetilde{P}_\gamma$
over $\mathbb{C}(t)$ is $A_n$.
As in \autoref{lemma:even-hilbert}, we follow \cite{sellares:realizing-sn-and-an-as-galois-groups} and correct a sign error.
Using \autoref{lemma:odd-disc-square}, $\Disc(\widetilde{P_\gamma})$ is a square, so the proof is reduced to showing that all the inertia groups are $3$-cycles and that the Galois group acts transitively on the $n$ roots of
$\widetilde{P}_\gamma$ over an algebraic closure.

By our discriminant computation in \autoref{lemma:odd-disc-square}, the ramified prime ideals are 
\[
(t), \ \ \ \Big(t\sqrt{\delta} \pm \sqrt{ (-1)^r \big(\frac{\partial \widetilde{P}}{\partial x}(\alpha) - \frac{\partial \widetilde{P}}{\partial x}(\beta_i)\big)}\Big).
\]
We now argue as in the even case. Modulo $(t)$, $\widetilde{P}$ has a double root at $x=\alpha$ and no other repeated roots; therefore, the corresponding
inertia group is either trivial or generated by a transposition and so must be trivial. Modulo any of the remaining ramified prime ideals, $\widetilde{P}$ has a triple root at $x = \beta_i$ and 
no other repeated roots, and the inertia group is trivial or cyclic of order $3$.

To complete the proof, we show the Galois group acts transitively on the $n$ roots of $\widetilde{P}_\gamma$ over an algebraic closure.
As a polynomial in $\mathbb{C}(x)[t]$, $\widetilde{P}_\gamma$ is reducible if and only if $(-1)^r\left(\frac{\partial \widetilde{P}}{\partial x}(\alpha)-\frac{1}{x}\widetilde{P}\right)$ is a square in $\mathbb{C}(x)$.  However, $\widetilde{P}(0) = -G(0) \neq 0$, so $\widetilde{P}$ is not divisible by $x$; consequently, $\widetilde{P}_\gamma$ is irreducible in $\mathbb{C}(x)[t]$.  Since $t$ appears only in the linear term in $\widetilde{P}_\gamma$ and since $x \nmid \widetilde{P}_\gamma$, we find that $\widetilde{P}_\gamma$ is irreducible in $\mathbb{C}(t)[x]$ as well and that the Galois group acts transitively on the $n$ roots of $\widetilde{P}_\gamma$ over an algebraic closure.

Lastly, as in the even case, this shows for any finite extension $L/K$ that the
Galois group of $\widetilde{P}_\gamma$ over $L(t)$ contains $A_n$.  The
remaining conclusions follow from \autoref{lemma:odd-disc-square} and the fact
that a degree $n$ extension whose Galois group contains $A_n$ is $A_n$ if and
only if it has square discriminant.
\end{proof}

Analogously to \autoref{lemma:even-big-galois}, we obtain the following:
\begin{lemma}
	Let $\delta \in K^\times$.  The polynomial $f_{(-1)^r \delta t^2 - \frac{\partial \widetilde{f} }{\partial x}(a)}(x) \in K(a_1,\dots,a_{r-1},a,t)[x]$ has Galois group $A_n$ if $\delta$ is a square and $S_n$ otherwise.  If $\delta$ is not a square, then $f_{(-1)^r \delta t^2 - \frac{\partial \widetilde{f} }{\partial x}(a)}(x)$ has Galois group $A_n$ over $K(\sqrt{\delta})(a_1,\dots,a_{r-1},a,t)$.
\end{lemma}
\begin{proof}
	Let $\gamma := (-1)^r \delta t^2 - \frac{\partial \widetilde{f} }{\partial x}(a)$.
	The Galois group of $f_{\gamma}(x)$ equals that of $\widetilde{f}_{\gamma}(x)$. Supposing first that $\delta$ is a square, this latter Galois group
	is contained in $A_n$ because
	the discriminant of $\widetilde{f}_{\gamma}(x)$ is a square by \autoref{lemma:odd-disc-square}, and it contains $A_n$ because the specialization of \autoref{lem:odd-specialization}
	has Galois group $A_n$.  The case that $\delta$ is not a square is analogous.
	\end{proof}

\subsubsection{Completing the proof of \autoref{thm:an-lower-bound} and \autoref{thm:sn-lower-bound}
for odd $n$}
\label{subsubsection:odd-proof}

For \autoref{thm:an-lower-bound}, we take $\delta =1$ and let our parameters vary over all integer values in the ranges \eqref{eq:bounds}.
We take $\gamma = (-1)^r t^2 - \frac{\partial \widetilde{f} }{\partial x}(a)$,
so $\| \tau \| \ll_\delta Y^{(n - 1)/2}$. 
We thus make $\gg Y^{\frac{d(n^2+7)}{8}}$ choices of these parameters.
Next, analogously to the case that $n$ is even described in
\autoref{subsubsection:even-proof},
for any fixed polynomial $q \in K[x]$, there are at most $\deg \gamma = 2$ possible values of $(\alpha_1, \ldots, \alpha_{r-1}, \alpha, \tau)$ so that 
$f_\gamma|_{\alpha_1, \ldots, \alpha_{r-1}, \alpha, \tau}= q$.

At this stage, we now proceed as in the even case. Applying \autoref{thm:hit}, we have
\[
	\#\{ \alpha_1, \dots, \alpha_{r-1}, \alpha, \tau \in \O_K : \mathrm{ht}(f|_{\alpha_1, \ldots, \alpha_{r-1}, \alpha, \tau}) \ll Y, \mathrm{Gal}(f|_{\alpha_1, \ldots, \alpha_{r-1}, \alpha, \tau} /K) \simeq A_n\}
	\gg Y^{\frac{d(n^2+7)}{8}}.
\]
Therefore, in the notation of \eqref{def:gk}, we have $\GnK{A_n}{Y} \gg_K Y^\frac{d(n^2+7)}{8}$,
and hence by taking $C = (n^2+7)/8$ in \autoref{cor:reduction-schmidt}
we conclude that  
\[
	\NnK{A_n}{X}
		\gg X^{\frac{(n - 7)(n + 2)}{8n^2}}.
\]
As in the even case, to obtain \autoref{thm:sn-lower-bound}, we instead take $\delta$ to be any integral element for which $M = K(\sqrt{\delta})$ and proceed as above.

\subsection{Minor improvements under stronger hypotheses}
As we observed in \autoref{prop:reduction}, 
our lower bounds can be improved slightly with improvements in the {\itshape upper} bounds,
which leads to the following result:

\begin{proposition}\label{prop:best-possible} \hspace{1in}
	\begin{enumerate}
		\item Let $n\geq 6$ be even and suppose that $\NnK{A_n}{X} \ll X^{\frac{n^2 - 2n + 8}{8n}}$.  Then 
\[
	\NnK{A_n}{X} \gg X^{\frac{n^2 - 2n + 8}{8(n^2 - n)}}.
\]
In particular, this holds unconditionally with $K = \Q$ for any $n \geq 430$.
\item Let $n \geq 7$ be odd and suppose that $\NnK{A_n}{X} \ll X^{\frac{n^2 - 4n + 7}{8n}}$.  Then
\[
	\NnK{A_n}{X} \gg X^{\frac{n^2 - 4n + 7}{8(n^2 -n)}}.
\]
In particular, this holds unconditionally  with $K = \Q$ for any $n \geq 433$.
	\end{enumerate}
\end{proposition}

We briefly indicate the idea behind \autoref{prop:best-possible} and omit a detailed proof.
Assume the stated upper bounds for $\NnK{A_n}{X}$ as in \autoref{prop:best-possible}(1) and (2).
These upper bounds were precisely constructed so that
$C
\geq n(e + 1/2)$ for $C$ and $e$ as in \eqref{eq:reduction}.
Then, following the same line of reasoning as in the proof of \autoref{thm:an-lower-bound}
(in particular, the reasoning in \autoref{subsubsection:even-proof} and \autoref{subsubsection:odd-proof})
leads to \autoref{prop:best-possible} after some elementary arithmetic.

Finally, the two hypothesized upper bounds on $N_{n, \mathbb{Q}}(A_n,X)$ hold for $n \geq 430$ and $n\geq 433$ by \cite[Theorem 1.2]{LOT} 
and an explicit computation (whose details we also omit). We may also obtain bounds for $K \neq \Q$ by counting degree $n$ extensions of $K$
as degree $n[K : \Q]$ extensions of $\Q$.


\appendix
\section{Hilbert Irreducibility}
\label{appendix:hilbert}

In the course of our proof of \autoref{thm:an-lower-bound}, we apply Hilbert irreducibility
to families of polynomials over a function field $K(a_1,\dots,a_r)$. 
We use a form of Hilbert irreducibility applied to counting polynomials in a box with varying edge lengths,
in a box which is not a hypercube.

Although it is well known to experts, we could not find an explicit statement of this particular form of Hilbert irreducibility.
For completeness, we include a proof
following the method of \cite[\S13]{Serre}.\footnote{
Serre does remark that his proof yields a uniform bound for the number of points
in {\itshape every} box of fixed diameter, which does suffice for our claimed statement, as was pointed out on
\cite[p. 733]{EllenbergVenkatesh}.}

\begin{definition}
	\label{definition:}
	For $K$ a number field, and $x \in K$, 
	as in \autoref{subsection:notation},
	define $\|x\| := \max_{\sigma} |\sigma(x)|$ as $\sigma$ ranges over all embeddings $K \rightarrow \mathbb C$ and $|\sigma(x)|$ denotes the complex norm.
For $S \subset \mathbb A^r(\O_K) = \O_K^r$, and positive real numbers $e_1, \ldots, e_r$, define 
\begin{align*}
S(T;e_1, \ldots, e_r) := \{ (a_1, \ldots, a_r) \in S : \|a_i\| \leq T^{e_i} \}
\end{align*}
and define
\begin{align*}
p_S(T;e_1, \ldots, e_r) := \frac{ \# S(T;e_1, \ldots, e_r)}{\# (\mathbb A^n(\O_K))(T;e_1, \ldots, e_r)}
\end{align*}
to be the proportion of points of $\mathbb A^n(\O_K))$ with $i$th coordinate
less than $T^{e_i}$ lying in $S$.
\end{definition}

To state the upcoming theorem, we now set some notation.
Let $K$ be a number field and let
	$F(a_1, \ldots, a_r, x) \in \O_K[a_1,\dots,a_r][x]$ be an element with Galois group $G$ when viewed as a polynomial in $x$ over $K(a_1,\dots,a_r)$. 
Let $S \subset \mathbb A^r(\O_K)$ denote the set of choices of integral elements $(\alpha_1, \ldots, \alpha_r) \in \mathbb A^r(\O_K)$ so that 
the image of $F$ in $\O_K[a_1,\dots,a_r][x]/(a_1 - \alpha_1, \ldots, a_r - \alpha_r)$
has Galois group $G$.
	
\begin{theorem}[Hilbert irreducibility]\label{thm:hit}
With notation as above, for $e_1, \dots, e_r \in \mathbb R_{> 0}$, we have
	$\lim_{T \rightarrow \infty} p_S(T;e_1, \ldots, e_r) = 1$.
\end{theorem}
\begin{proof}
By \cite[\S9.2, Proposition 2]{Serre}, the set of exceptions $(a_1, \dots, a_r) \in \mathbb A^r(\O_K) = \O_K^r$ belong to a {\itshape thin set}. 
Recall that a thin set in $\mathbb A^r$ can be described geometrically as a subset $\Omega \subset \mathbb A^r(K)$ so that there exists
some generically quasi-finite $\pi\colon X \rightarrow \mathbb A^r$ with $\Omega \subset \pi(X(K))$ and 
so that each irreducible component of $X$ which dominates $\mathbb A^r$ maps to $\mathbb A^r$ with degree at least $2$.

Hence, it suffices to prove that for $M$ the intersection of a thin set with 
$\O_K^r$, 
\begin{align*}
\lim_{T \rightarrow \infty} p_M(T;e_1, \ldots, e_r) = 0.
\end{align*}
Let $M = \O_K^r \cap \pi(X(K))$ be our specified thin set, for $\pi: X \rightarrow \mathbb A^r$ as in the definition of thin set above.
If $M$ is contained in the image of $X(K)$ it suffices to prove the statement for each of the irreducible components of $X$ separately, and we henceforth assume $X$ is irreducible.

We first consider the more difficult case when $X$ dominates $\mathbb A^r$, in which case $X \rightarrow \mathbb A^r$ has degree at least $2$.
For $\mathfrak p \subset \O_K$ a prime, let $M_{\mathfrak p}$ denote the reduction $M \bmod \mathfrak p$, viewed as a subset of $(\O_K/\mathfrak p)^r$, and let
$N(\mathfrak p) \in \mathbb Z$ denote the norm of the ideal $\mathfrak p$.
By \cite[Theorem 5, \S13.2]{Serre}, there is a finite Galois extension $K_\pi/K$ and a constant $c_{\pi} < 1$ with the following property:
For each prime $\mathfrak p \subset \O_K$ which splits completely in
$K_\pi$,  we have $\# M_{\mathfrak p} \leq c_\pi N(\mathfrak p)^r + O(N(\mathfrak p)^{r - 1/2})$.
In particular, for all primes $\mathfrak p$ of sufficiently large norm which split
completely in $K_\pi$, the image $M_{\mathfrak p}$ in $(\O_K/\mathfrak p)^r$
has density 
$\delta_{\mathfrak p}
=
\frac{c_\pi N(\mathfrak p)^r + O(N(\mathfrak p)^{r - 1/2})}{N(\mathfrak p)^r}$, 
which is bounded away from $1$.

Let $\mathcal{S}$ be the set of such primes $\mathfrak p$ which are sufficiently large in the above sense and which split completely in $K_\pi$. For any finite subset $\mathcal{S}' \subseteq \mathcal{S}$,
it follows from the Chinese remainder theorem that
$p_M(T;e_1, \ldots, e_r)$ is bounded above by
\begin{equation}\label{eq:chebo_setup}
	\prod_{\mathfrak p \in \mathcal{S}'} \delta_{\mathfrak p} + o_{T, \mathcal{S}'}(1).
\end{equation}
Since $\mathcal{S}$ contains infinitely many primes by the Chebotarev density theorem, and each $\delta_{\mathfrak p}$ is bounded away from $1$, 
the product in \eqref{eq:chebo_setup} may be taken 
arbitrarily close to zero, proving the theorem in the case that $X \rightarrow \mathbb A^r$ is dominant.

If $\pi\colon X \rightarrow \mathbb A^r$ is not dominant, 
then $\pi(X)$ must instead be contained in some Zariski closed subset of $\mathbb A^r$,
so it suffices to deal with the case that $X \subset \mathbb A^r$ is Zariski closed.
The proof in this case is analogous to the case that $\pi\colon X \rightarrow \mathbb A^r$ is dominant,
and we even obtain the stronger bound that $M_{\mathfrak p}$ has at most $N(\mathfrak p)^{r-1} + O(N(\mathfrak p)^{r - 3/2})$ elements.
The rest of the argument then goes through analogously, since the associated 
densities
$\delta_{\mathfrak p} := \frac{N(\mathfrak p)^{r-1} + O(N(\mathfrak p)^{r - 3/2})}{N(\mathfrak p)^r}$ 
satisfy $\delta_{\mathfrak p} < 1$ for all primes $\mathfrak p$ of sufficiently
large norm.
\end{proof}

\bibliographystyle{alpha}
\bibliography{an-refs}

\newcommand{\etalchar}[1]{$^{#1}$}
\begin{thebibliography}{PTBW20}

\bibitem[Bai80]{baily}
Andrew~Marc Baily.
\newblock On the density of discriminants of quartic fields.
\newblock {\em J. Reine Angew. Math.}, 315:190--210, 1980.

\bibitem[BCT]{BCT}
Manjul Bhargava, Alina Cojocaru, and Frank Thorne.
\newblock The number of non-${S}_5$-quintic extensions of bounded discriminant.
\newblock {\em In preparation}.

\bibitem[BST{\etalchar{+}}20]{BSTTTZ}
M.~Bhargava, A.~Shankar, T.~Taniguchi, F.~Thorne, J.~Tsimerman, and Y.~Zhao.
\newblock Bounds on 2-torsion in class groups of number fields and integral
  points on elliptic curves.
\newblock {\em J. Amer. Math. Soc.}, 33(4):1087--1099, 2020.

\bibitem[Coh54]{cohn}
Harvey Cohn.
\newblock The density of abelian cubic fields.
\newblock {\em Proc. Amer. Math. Soc.}, 5:476--477, 1954.

\bibitem[Cou20]{couveignes2019enumerating}
Jean-Marc Couveignes.
\newblock Enumerating number fields.
\newblock {\em Ann. of Math. (2)}, 192(2):487--497, 2020.

\bibitem[EV06]{EllenbergVenkatesh}
Jordan~S. Ellenberg and Akshay Venkatesh.
\newblock The number of extensions of a number field with fixed degree and
  bounded discriminant.
\newblock {\em Ann. of Math. (2)}, 163(2):723--741, 2006.

\bibitem[Hil92]{Hilbert}
David Hilbert.
\newblock Ueber die {I}rreducibilit\"{a}t ganzer rationaler {F}unctionen mit
  ganzzahligen {C}oefficienten.
\newblock {\em J. Reine Angew. Math.}, 110:104--129, 1892.

\bibitem[Lan02]{Lang}
Serge Lang.
\newblock {\em Algebra}, volume 211 of {\em Graduate Texts in Mathematics}.
\newblock Springer-Verlag, New York, third edition, 2002.

\bibitem[LR13]{LarsonRolen}
Eric Larson and Larry Rolen.
\newblock Upper bounds for the number of number fields with alternating
  {G}alois group.
\newblock {\em Proc. Amer. Math. Soc.}, 141(2):499--503, 2013.

\bibitem[LT19]{LemkeOliverThorne}
Robert~J. {Lemke Oliver} and Frank {Thorne}.
\newblock {Rank growth of elliptic curves in non-abelian extensions}.
\newblock {\em Int. Math. Res. Not. IMRN}, 2019.
\newblock Art. {ID} rnz307.

\bibitem[LT20]{LOT}
Robert~J. {Lemke Oliver} and Frank {Thorne}.
\newblock {Upper bounds on number fields of given degree and bounded
  discriminant}.
\newblock {\em Preprint}, 2020.
\newblock Available at \url{https://arxiv.org/abs/2005.14110}.

\bibitem[Mal04]{malle}
Gunter Malle.
\newblock On the distribution of {G}alois groups. {II}.
\newblock {\em Experiment. Math.}, 13(2):129--135, 2004.

\bibitem[MiS17]{sellares:realizing-sn-and-an-as-galois-groups}
Mireia Mart\'inez~i Sellar\`es.
\newblock Realizing {$S_n$} and {$A_n$} as {G}alois {G}roups over {$\Q$}: {A}n
  {I}ntroduction to the {I}nverse {G}alois {P}roblem.
\newblock {\em Undergraduate thesis, {U}niversitat de {B}arcelona}, 2017.
\newblock Available at
  \url{http://diposit.ub.edu/dspace/bitstream/2445/113081/1/memoria.pdf}.

\bibitem[PTBW20]{PTBW}
Lillian~B. Pierce, Caroline~L. Turnage-Butterbaugh, and Melanie~Matchett Wood.
\newblock An effective {C}hebotarev density theorem for families of number
  fields, with an application to {$\ell$}-torsion in class groups.
\newblock {\em Invent. Math.}, 219(2):701--778, 2020.

\bibitem[Sch95]{schmidt}
Wolfgang~M. Schmidt.
\newblock Number fields of given degree and bounded discriminant.
\newblock {\em Ast\'erisque}, (228):4, 189--195, 1995.
\newblock Columbia University Number Theory Seminar (New York, 1992).

\bibitem[Ser97]{Serre}
Jean-Pierre Serre.
\newblock {\em Lectures on the {M}ordell-{W}eil theorem}.
\newblock Aspects of Mathematics. Friedr. Vieweg \& Sohn, Braunschweig, third
  edition, 1997.
\newblock Translated from the French and edited by Martin Brown from notes by
  Michel Waldschmidt, with a foreword by Brown and Serre.

\bibitem[Wri89]{wright}
David~J. Wright.
\newblock Distribution of discriminants of abelian extensions.
\newblock {\em Proc. London Math. Soc. (3)}, 58(1):17--50, 1989.

\end{thebibliography}

\end{document}